\documentclass[11pt,english,reqno]{amsart}
\usepackage[T1]{fontenc}
\usepackage[latin9]{inputenc}
\usepackage{amsthm}
\usepackage{amssymb}

\makeatletter
\numberwithin{equation}{section}
\numberwithin{figure}{section}
\theoremstyle{plain}
\newtheorem{thm}{Theorem}
  \theoremstyle{plain}
  \newtheorem{prop}[thm]{Proposition}

\usepackage[matrix,arrow,tips,curve,ps]{xy}\SelectTips{cm}{10}
\usepackage{enumerate}
\usepackage{comment} 

\newtheorem{Theorem}{Theorem}

\theoremstyle{definition}

\theoremstyle{Example}
 \newtheorem{example} [Theorem] {Example}
 
\theoremstyle{definition}
  \newtheorem{defn}[thm]{Definition}
   \theoremstyle{plain}
  
  \theoremstyle{corollary}
  \newtheorem{cor}[thm]{Corollary}
\theoremstyle{remark}
  \newtheorem{remark}[thm]{Remark}
 
\numberwithin{Theorem}{section} \numberwithin{equation}{section}

\def\bC{{\mathbb C}}

\def\bP{{\mathbb P}}

\def\bN{{\mathbb N}}
\def\bQ{{\mathbb Q}}

\makeatother

\usepackage{babel}

\begin{document}

\newcommand{\Gal}{{\rm Gal}} 
\newcommand{\Pic}{{\rm Pic}} 
\def\bP{{\mathbb P}}

\title[Canonical surfaces with big cotangent bundle]{Canonical surfaces with big cotangent bundle}

\addtolength{\textwidth}{0mm}
\addtolength{\hoffset}{-0mm} 
\addtolength{\textheight}{0mm}
\addtolength{\voffset}{-0mm}

\author{Xavier Roulleau, Erwan Rousseau}
\thanks{The second author is partially supported by the ANR project \lq\lq POSITIVE\rq\rq{}, ANR-2010-BLAN-0119-01.}
\subjclass[2010]{Primary: 14J60, 14J70; Secondary: 14J25, 32Q45.}
\date{}

\begin{abstract}
Surfaces of general type with positive second Segre number
are known to have big cotangent bundle. We give a new criterion ensuring
that a surface of general type with canonical singularities has a minimal resolution
with big cotangent bundle. This provides many examples of surfaces with negative second Segre number and big cotangent bundle.
\end{abstract}

\maketitle

\section{Introduction}
Projective varieties with positive cotangent bundle have attracted a lot of attention because of
the strong geometric properties they have. In particular, they are Kobayashi hyperbolic and algebraically hyperbolic (see \cite{Debarre} for an introduction).

Surfaces of general type with ample cotangent bundle are known to have positive second Segre class $s_2:=c_1^2-c_2 >0$ \cite{FL}. In fact, Bogomolov \cite{Bogo77} proved that if $X$ is a surface of general type with positive second Segre class, then the family of curves on $X$ of fixed geometric genus is bounded. The numerical positivity ensures that these surfaces have many symmetric tensors, in other words their cotangent bundle is big.

Besides surfaces with positive second Segre number, it seems that there are few examples of surfaces with big cotangent bundle. Smooth surfaces in $\bP^3$ are well known to have no symmetric tensors \cite{Sakai}, but Bogomolov and De Oliveira \cite{BogO} made the interesting observation that minimal resolutions of singular surfaces may provide such examples. In \cite{BogO} nodal surfaces are considered and a numerical condition on the number of nodes is given to ensure that the resolution will have big cotangent bundle. Unfortunately, this statement turns out not to be completely correct (see section \ref{nodal} below for details).

In this article, we would like to give a new general criterion ensuring that a surface of general type has big cotangent bundle. Let us describe our result.

Recall that if we take a minimal model $Y$ of a surface of general type (i.e. smooth with $K_Y$ nef and big) then the curves $E$ with $K_Y.E=0$ form bunches of $(-2)$-curves, and can be contracted to \emph{canonical} singularities (also known as $ADE$ or Du Val singularities in the case of surfaces, see for example \cite{Reid} or \cite{KoMo}).

Let $X$ be a \emph{canonical} surface i.e a projective surface with positive canonical divisor $K_X$ and at worst canonical singularities. In dimension $2$, canonical singularities are known to be quotients of $\bC^2$ by finite subgroups of $SL(2, \bC).$ Therefore we can attach two objects to the surface $X$.
On the one hand, we consider $Y\to X$ its minimal resolution. On the other hand, we let $\mathcal{X} \to X$ to be the orbifold (or stack) attached to $X$.

We denote by $s_{2}(Y)=c_{1}^{2}(Y)-c_{2}(Y)$
and $s_{2}(\mathcal{X})=c_{1}^{2}(\mathcal{X})-c_{2}(\mathcal{X})$
the second Segre numbers of $Y$ and $\mathcal{X}$ respectively.

\begin{thm}\label{main}
Let $X$ be a canonical surface, $Y\to X$ its minimal resolution and $\mathcal{X} \to X$ the orbifold associated to $X$. If $$s_{2}(Y)+s_{2}(\mathcal{X})>0,$$
then the cotangent bundle of $Y$ is big. In particular, $Y$ has only finitely many rational or elliptic curves.
\end{thm}

Thanks to the work of McQuillan \cite{McQ0}, this result also provides surfaces of general type satisfying the Green-Griffiths-Lang conjecture, since entire curves in such surfaces will be contained in a proper algebraic subvariety.

As applications, we obtain many examples of surfaces with big cotangent bundle and negative second Segre number. Among them, we have the following results.

\begin{thm}\label{A_kthm}
Let $X \subset \bP^3$ be a hypersurface of degree $d$ with $\ell$ singularities $A_k$ and let $Y \to X$ be its minimal resolution. If
$$ \ell > \frac{4(k+1)}{k(k+2)} (2d^2-5d),$$
then $Y$ has a big cotangent bundle.
\end{thm}

As a consequence, we obtain the existence of surfaces with big cotangent sheaf for all degrees $d \geq 13$. In the following  application of Theorem \ref{A_kthm}, the singularities of the surface considered are $A_{d-1}$:

\begin{thm}\label{ram}
Let $X \subset \bP^3$ be the ramified cover of $\bP^2$ of degree $d=\sum d_j$, branched along the normal crossing divisor $\ensuremath{D=\cup_{j=1}^{j=k}D_{j}\subset\bP^{2}}$, where $D_{j}$
is a curve of degree $d_{j}$.
Suppose that $d_i \geq c$ for $i=1,\dots,k$ and
$$k(k-1) > \frac{8d^2(2d-5)}{c^2(d^2-1)},$$
then the minimal resolution $Y \to X$ has big cotangent bundle.
\end{thm}

This result provides examples of hypersurfaces with big cotangent bundle for all degrees $d \geq 15$. 

\

The paper is structured as follows. In section $2$ we present the orbifold setting and the results we will use. In section $3$ we prove our main Theorem \ref{main}. In section $4$ we give applications proving Theorem \ref{A_kthm} and Theorem \ref{ram}. Section $5$ is devoted to the geography of surfaces of general type with big cotangent bundle.

\subsubsection*{Acknowledgments}
The second author thanks Fedor Bogomolov, Bruno De Oliveira and Jordan Thomas for several interesting and fruitful discussions and remarks. 
He also thanks the CNRS for the opportunity to spend a semester in Montreal at the UMI CNRS-CRM and the hospitality of UQAM-Cirget where part of this work was done.

\section{Orbifolds basics}

For the convenience of the reader, we recall in this section the basic facts on orbifolds (refering to \cite{RR} for details) that will be used in the proof of Theorem \ref{main}.

\subsection{Orbifolds and canonical singularities}
We define orbifolds as a particular type of log pairs. The data $(X,\Delta)$ is a log pair if $X$ is a normal algebraic variety (or a normal complex space) and $\Delta=\sum_id_iD_i$ is an effective $\bQ$-divisor where the $D_i$ are distinct, irreducible divisors and $d_i \in \bQ$. 

For orbifolds, we need to consider only pairs $(X,\Delta)$ such that $\Delta$ has the form
$\Delta=\sum_i \left(1-\frac{1}{m_i}\right) D_i,$
where the $D_i$ are prime divisors and $m_i \in \bN^*$. 

\begin{defn}
An {\em orbifold chart} on $X$ compatible with $\Delta$ is a Galois covering $\varphi : U \to \phi(U)\subset X$ such that
\begin{enumerate}
\item $U$ is a domain in $\bC^n$ and $\varphi(U)$ is open in $X$,
\item the branch locus of $\varphi$ is $\lceil \Delta \rceil \cap \varphi(U)$,
\item for any $x \in U'':=U\setminus \varphi^{-1}(X_{sing}\cup \Delta_{sing})$ such that $\varphi(x) \in D_i$, the ramification order of $\varphi$ at $x$ verifies $ord_{\varphi}(x)=m_i.$
\end{enumerate}
\end{defn}

\begin{defn}
An orbifold $\mathcal{X}$ is a log pair $(X,\Delta)$ such that $X$ is covered by orbifold charts compatible with $\Delta.$
\end{defn}

\begin{defn}
Let $(X,\Delta)$, $\Delta=\sum_i\left(1-\frac{1}{m_i}\right)C_i$, be a pair where $X$ is a normal surface and $K_X+\Delta$ is $\mathbb{Q}$-Cartier. Let $\pi: \tilde{X} \to X$ be a resolution of the singularities of $(X,\Delta)$, so that the exceptional divisors, $E_i$ and the components of $\tilde{\Delta}$, the strict transform of $\Delta$, have normal crossings and
$K_{\tilde{X}}+\tilde{\Delta}=\pi^*(K_X+\Delta)+\sum_ia_iE_i.$

We say that $(X,\Delta)$ has \emph{canonical} singularities if $a_i \geq 0$ for every exceptional curve $E_i$.

If $\Delta=0$, then canonical singularities of $X$ are the same as Du Val singularities (or $ADE$ singularities), which are quotient singularities (see \cite{KoMo} for details). As a consequence $X$ has an orbifold structure $\mathcal{X}$. Moreover, their minimal resolution $Y\to X$ is such that $K_Y=f^*K_X.$
\end{defn}

\subsection{Chern classes}
Let $\pi: \mathcal{X} \to (X, \Delta)$ be a two dimensional orbifold for which  $\Delta=0$ and the singularities are $ADE$.
 Let $a_{n}$ (resp $d_{n}$, $e_{n}$) be the number of $A_{n}$ (resp. $D_{n},\, E_{n}$)
singularities on $X$ and $Y \to X$ its minimal resolution.

\begin{prop}[\cite{RR}]\label{comp}
\label{ChernNumberOrbi}
The Chern numbers of $\mathcal{X}$ are $c_{1}^{2}(\mathcal{X})=c_{1}^{2}(X)=c_{1}^{2}(Y)$ and:
\begin{equation}
\label{Chernnumb}
c_{2}(\mathcal{X})=c_{2}(Y)-\sum(n+1)(a_{n}+d_{n}+e_{n})+\sum\frac{a_{n}}{n+1}+\frac{d_{n}}{4(n-2)}+\frac{e_{6}}{24}+\frac{e_{7}}{48}+\frac{e_{8}}{120}.
\end{equation}
\end{prop}

The denominators $4(n-2)$, $24,\,48,\,120$ are the order of the
binary dihedral $BD_{4(n-2)}$, the binary tetrahedral,
the binary octahedral, and the binary icosahedral group respectively.

\subsection{Orbifold Riemann-Roch}

Let $L$ be an orbifold line bundle on the orbifold $\mathcal{X}$ of dimension $n$. We will use Kawazaki's orbifold Riemann-Roch theorem \cite{Kawa} or To\"en's for Deligne-Mumford stacks \cite{Toen} using intersection theory on stacks.

\begin{thm}[\cite{Toen}]\label{Roch}
Let $\mathcal{X}$ be a Deligne-Mumford stack with quasi-projective coarse moduli space and which has the resolution property (i.e every coherent sheaf is a quotient of a vector bundle). Let $E$ be a coherent sheaf on $\mathcal{X}$ then
$$\chi(\mathcal{X},E)=\int_{\mathcal{X}}\widetilde{ch}(E)\widetilde{Td}(T_{\mathcal{X}}).$$
\end{thm}

From this, we obtain the asymptotic formula:
$$\chi(\mathcal{X}, L^k)=\frac{c_1(L)^n}{n!}k^n+O(k^{n-1}),$$
using orbifold Chern classes.

We will apply this result to orbifold surfaces $\mathcal{X}$ of general type associated to canonical surfaces. 
Then
$$\chi(\mathcal{X},S^{m}\Omega_{\mathcal{X}})=\frac{m^3}{6}(c_{1}^2-c_{2})+O(m^2),$$
 where $c_{1}$ and $c_{2}$ are the orbifold Chern classes of $\mathcal{X}$.

\subsection{Vanishing theorems}
In the case of smooth minimal surfaces of general type $Y$, thanks to the semi-stability of the cotangent bundle $\Omega_{Y}$ with respect to $K_{Y}$, we have Bogomolov's vanishing Theorem \cite{Bogo79}:
\begin{equation}\label{vanish1}
H^0(Y, S^mT_{Y}\otimes K_{Y}^p)=0,
\end{equation}
for $m-2p>0.$

Now, let us consider an orbifold surface $\mathcal{X}$ of general type associated to a canonical surface $X$. Then Bogomolov's vanishing Theorem easily extends to this situation:
\begin{equation}\label{vanish2}
H^0(\mathcal{X}, S^mT_{\mathcal{X}}\otimes K_{\mathcal{X}}^p)=0,
\end{equation}
for $m-2p>0.$

Indeed, $\mathcal{X}$ can be equipped with an orbifold K\"ahler-Einstein metric \cite{TY} (or \cite{Kob85}) and the standard Bochner identities apply \cite{Kob80}.

As a corollary, if $s_2(\mathcal{X}):=c_1^2(\mathcal{X})-c_2(\mathcal{X})>0$ then  $$H^0(\mathcal{X}, S^m\Omega_{\mathcal{X}}) \geq \frac{s_2(\mathcal{X})}{6}m^3+ O(m^2).$$

\subsection{Logarithmic differentials and extension of sections}
Let $X$ be a complex manifold with a normal crossing
divisor $D.$ 

The logarithmic
cotangent sheaf $\Omega_X(\log D)$ is defined as the locally free subsheaf of the
sheaf of
meromorphic 1-forms on $X,$ whose restriction to $X \setminus {D}$ is $\Omega_{X}$ and whose localization at any point $x\in D$ is
given by
\begin{equation*}
\Omega_X(\log D)_x=\underset{i=1}{\overset{l}{\sum }}\mathcal{O}_{%
X,x}\frac{dz_{i}}{z_{i}}+\underset{j=l+1}{\overset{n}{\sum }}%
\mathcal{O}_{X,x}dz_{j}
\end{equation*}
where the local coordinates $z_{1,}...,z_{n}$ around $x$ are
chosen such that $D=\{$ $z_{1}...z_{l}=0\}.$

Let $X$ be a projective surface with canonical singularities, $Y \to X$ the minimal resolution with $E$ the exceptional divisor and $\mathcal{X} \to X$ the orbifold.

Sections of $H^0(\mathcal{X}, S^m\Omega_{\mathcal{X}})$ do not give a priori sections of $H^0(Y, S^m\Omega_Y)$, but only sections of $H^0(Y \setminus E, S^m\Omega_Y)$. Nevertheless, in the case of quotient singularities, we have the following extension theorem of Miyaoka \cite{Miy} (see also \cite{GKK}):
\begin{equation}\label{extension}
H^0(Y \setminus E, S^m\Omega_Y)\cong H^0(Y, S^m\Omega_Y(\log E)).
\end{equation}

\section{Proof of Theorem \ref{main}}

In this section we prove Theorem \ref{main}, the main result of this paper:

\begin{thm}

Suppose $s_2(Y)+s_2(\mathcal X )>0$. Then 
$$h^0 (Y,S^m \Omega_Y)\geq  \frac{s_2(Y)+s_2(\mathcal X )}{12}\,m^3+O(m^2),$$
in particular the cotangent bundle of $Y$ is big.

\end{thm}

For $m\in\mathbb{N}^{*}$, let us consider the following exact sequence:

\begin{equation}\label{eq:1}
0\to S^m\Omega_{Y}\to S^m\Omega_{Y}(\log E)\to Q_{m} \to 0
\end{equation}

The quotient sheaf $Q_{m}$ is supported by the divisor $E$ that
is the sum of the exceptional components of the map $Y\to X$. Since
the singularities of $X$ are $ADE$, there exists a neighborhood $U$
of $E$ such that the canonical sheaf of $Y$is trivial : $(K_{Y})_{|U}\simeq(\mathcal{O}_{Y})_{|U}$.
Therefore multiplying by $\otimes K_{Y}^{\otimes(1-m)}$, we get the
following exact sequence: 

\begin{equation}\label{eq:2}
0\to S^m\Omega_{Y}\otimes K_{Y}^{\otimes(1-m)}\to S^m\Omega_{Y}(\log E)\otimes K_{Y}^{\otimes(1-m)}\to Q_{m} \to 0.
\end{equation}

The proof will distinguish two cases according to the value of $\varlimsup\frac{h^{0}(Q_{m})}{m^{3}}$.

\

Let us first suppose that \[
\varlimsup\frac{h^{0}(Q_{m})}{m^{3}}\leq\frac{s_{2}(\mathcal{X})-s_{2}(Y)}{12}.\]

As explained above, the Riemann-Roch Theorem \ref{Roch} and Bogomolov's vanishing property \ref{vanish2} give 
$$
\varlimsup\frac{1}{m^{3}}h^{0}(\mathcal{X},S^m\Omega_{\mathcal{X}})\geq\frac{s_{2}(\mathcal{X})}{6}.
$$

Combined with the extension property \ref{extension}, this implies that
$$
\varlimsup\frac{1}{m^{3}}h^{0}(Y,S^m\Omega_{Y}(\log E))\geq\frac{s_{2}(\mathcal{X})}{6}.
$$

Then the exact sequence \ref{eq:1} implies:\[
\begin{array}{cl}
\varlimsup\frac{h^{0}(S^m \Omega_{Y})}{m^{3}}\geq \varlimsup\frac{1}{m^{3}}h^{0}(Y,S^m\Omega_{Y}(\log E)) -  \varlimsup\frac{h^{0}(Q_{m})}{m^{3}} & \geq\frac{s_{2}(\mathcal{X})}{6}-\frac{s_{2}(\mathcal{X})-s_{2}(Y)}{12}\\
 & =\frac{s_{2}(\mathcal{X})+s_{2}(Y)}{12}.\end{array}\]

\

Let us suppose now that \[
\varlimsup\frac{h^{0}(Q_{m})}{m^{3}}>\frac{s_{2}(\mathcal{X})-s_{2}(Y)}{12}.\]

The extension property \ref{extension}, combined with the triviality of $K_Y$ on $U$, and Serre duality gives 
\[
\begin{array}{cl}
h^{0}(Y,S^m\Omega_{Y}(\log E)\otimes K_{Y}^{\otimes(1-m)})  & \cong h^{0}(Y\setminus E,S^m\Omega_{Y}\otimes K_{Y}^{\otimes(1-m)})\\
& \cong h^0(\mathcal{X},S^m\Omega_{\mathcal{X}}\otimes K_{\mathcal{X}}^{\otimes(1-m)})\\
& \cong h^{2}(\mathcal{X},S^m\Omega_{\mathcal{X}}).
\end{array}
\]

The latter dimension being zero by Bogomolov's vanishing property \ref{vanish2}, we obtain 
$$
h^{0}(Y,S^m\Omega_{Y}(\log E)\otimes K_{Y}^{\otimes(1-m)})=0.
$$

Thus, by the exact sequence \ref{eq:2}, we obtain

\[
h^{0}(Q_{m})\leq h^{1}(Y,S^m\Omega_{Y}\otimes K_{Y}^{\otimes(1-m)}).
\]

Serre duality again implies that
$$
h^{1}(Y,S^m\Omega_{Y}\otimes K_{Y}^{\otimes(1-m)})=h^{1}(Y,S^m\Omega_{Y}).
$$
Since $h^{2}(Y,S^m\Omega_{Y})=0$ by Bogomolov's vanishing property \ref{vanish1}, we get by Riemann-Roch: 
\[
\varlimsup\frac{1}{m^{3}}h^{0}(Y,S^m\Omega_{Y})=\varlimsup\frac{1}{m^{3}}(\chi(S^m\Omega_{Y})+h^{1}(Y,S^m\Omega_{Y}))\geq\frac{s_{2}(Y)}{6}+\frac{s_{2}(\mathcal{X})-s_{2}(Y)}{12}\]
and therefore\[
\varlimsup\frac{1}{m^{3}}h^{0}(Y,S^m\Omega_{Y})\geq\frac{s_{2}(\mathcal{X})+s_{2}(Y)}{12}.\]
 In any of the two above cases, we get \[
\varlimsup\frac{1}{m^{3}}h^{0}(Y,S^m\Omega_{Y})\geq\frac{s_{2}(\mathcal{X})+s_{2}(Y)}{12},\]
and therefore the cotangent sheaf of $Y$ is big.

\section{Applications}
\subsection{Surfaces with $A_k$ singularities}\label{nodal}

\label{nodal} 
As a corollary of Theorem
\ref{main} we obtain:

\begin{thm} \label{CriteriaNodes}
Let $X \subset \bP^3$ be a hypersurface of degree $d$ with $\ell$ singularities $A_k$ and let $Y \to X$ be its minimal resolution. If
$$ \ell > \frac{4(k+1)}{k(k+2)} (2d^2-5d),$$
then $Y$ has a big cotangent bundle.

 \end{thm}
 
\begin{proof}
Applying Proposition \ref{comp} and Brieskorn resolution theorem \cite{Brieskorn}, easy computations give \[
s_{2}(Y)=10d-4d^{2},\, s_{2}(\mathcal{X})=10d-4d^{2}+(k+1-\frac{1}{k+1})\, \ell .\]
 We apply Theorem \ref{main} with these values.
  \end{proof}
 
 
\begin{cor} \label{Surface13} 
If $d\geq13$ there exists nodal surfaces in $\bP^{3}$ of degree
$d$ whose minimal resolution has big cotangent bundle.
\end{cor}

\begin{proof}
The condition on the number $\ell$ of nodes is $\ell>\frac{8}{3}(2d^{2}-5d)$.
In \cite{Segre} Segre constructed nodal hypersurfaces with $\ell\geq\frac{1}{4}d^{2}(d-1)$
nodes. For $d\geq20$, we have $\frac{1}{4}d^{2}(d-1)>\frac{8}{3}(2d^{2}-5d)$, thus
we obtain examples of hypersurfaces of degree $d\geq20$ with symmetric
differentials.

Chmutov (see \cite[p.58]{Labs} or \cite{Chmutov}) constructed surfaces of degree $d$
with the number $\mu(d)$ of $A_{1}$ singularities  as follows:
\vspace{0.1cm}

\hspace{1cm}
\begin{tabular}{|c|c|c|c|c|c|c|c|}
\hline 
$d$  & $13$  & $14$  & $15$  & $16$  & $17$  & $18$  & $19$ \tabularnewline
\hline 
$\mu(d)$  & $732$  & $949$  & $1155$  & $1450$  & $1728$  & $2097$  & $2457$ \tabularnewline
\hline 
$[\frac{8}{3}(2d^{2}-5d)]+1$  & $729$  & $859$  & $1001$  & $1153$  & $1315$  & $1489$  & $1673$ \tabularnewline
\hline
\end{tabular}

\vspace{0.1cm}

\noindent Therefore we also obtain examples for $d$ in the range $13\leq d\leq19$.
\end{proof}


\begin{remark}
In \cite{BogO} the result of Corollary \ref{Surface13} was claimed for hypersurfaces of degree $d\geq 6$.
However, the proof uses the results of \cite[Lemma  2.2]{BogO} which
turns out to be false. Let us explain this
briefly in more details, using the notations of the proof of Theorem
\ref{main}. In Lemma 2.2, p. 94 of \cite{BogO}, it is claimed that
\[
\dim\left(H^{0}(Y\setminus E,S^{m}\Omega_{Y})/H^{0}(Y,S^{m}\Omega_{Y})\right)=\frac{1}{4} \, \ell \, m^{3}+O(m^{2})\]
where $\ell$ is the number of nodal singularities on $X$.

 To compute this dimension, the authors exhibit symmetric differentials of $Y \setminus E$
non zero in the quotient but do not verify the linear independence.
In fact, one can verify that \[
\dim\left(H^{0}(Y\setminus E,S^{m}\Omega_{Y})/H^{0}(Y,S^{m}\Omega_{Y})\right)= \frac{11}{108}\ell \, m^{3}+O(m^{2}).\]
 
This computation is done independently in \cite{Thomas}, deriving slightly better bounds in the case of nodes. In particular, the existence of a surface of degree $10$ with big cotangent sheaf is obtained.
\end{remark}

\subsection{Ramified covers of the plane}
In this section we give applications of Theorem \ref{main} for cyclic covers of the plane.

Let $\ensuremath{D=\cup_{j=1}^{j=k}D_{j}\subset\bP^{2}}$, where $D_{j}$
is a smooth curve of degree $d_{j}$ and such that $D$ has nodal
singularities. For $n>1$ dividing $d=\sum_{j=1}^{j=k}d_{j}$, there
exists a $n$-cyclic covering $X\to\bP^{2}$ branched along $D$. 

Since locally a singularity $s$ of $D$ has equation $x^{2}+y^{2}=0$,
the singularity in $X$ above $s$ has equation $z^{n}=x^{2}+y^{2}$
and is a $A_{n-1}$ singularity. 

The Chern numbers of the desingularization $Y$ of $X$ are:\[
\begin{array}{cc}
c_{1}^{2}= & n(-3+(1-\frac{1}{n})d)^{2}\\
c_{2}= & 3n+(n-1)(d^{2}-3d),\end{array}\]
and $Y$ has general type unless $(d,n)=(2,2),\,(4,2),\,(6,2),\,(3,3),\,(4,4)$,
cases we disregard from now on. We remark that the surface $Y$ is minimal. Such ramified coverings provide a family where the number of symmetric differentials may jump.

First, one should notice that in the smooth case there is no symmetric differentials at all:
\begin{prop}
Suppose that  $X$ is smooth. Then: 
$$H^0(X,S^m\Omega_X)=0,$$\,
for all $m>0$.
\end{prop}

\begin{proof}
Let us denote by $W\to\mathbb{P}^{2}$ the cyclic degree $d$
cover branched over the smooth curve $D$. There is a cyclic degree $v=\frac{d}{n}$
cover $g:W\to X$ making the following diagram\[
\begin{array}{ccc}
W & \stackrel{g}{\dashrightarrow} & X\\
 & \searrow & \downarrow\\
 &  & \mathbb{P}^{2}\end{array}\]
commute. Since $W$ is a smooth hypersurface in $\bP^3$ the space $H^{0}(W,S^{m}\Omega_{W})$ is $0$ \cite{Sakai}. That implies  $H^{0}(X,S^{m}\Omega_{X})=0$.
\end{proof}
\subsubsection{Criteria for $n=d$ and arbitrary $D_j$}



 Let us consider the case when the cover has degree $n=d$. 
This gives us a hypersurface $X \subset \bP^3$ of degree $d$ with $A_{d-1}$ singularities
over the singularities of $D$ i.e. the intersection points of the $D_{j}$'s.

\begin{thm}
Suppose that $d_i \geq c$ for $i=1,\dots,k$ and
$$k(k-1) > \frac{8d^2(2d-5)}{c^2(d^2-1)},$$
then the minimal resolution $Y \to X$ has big cotangent bundle.
\end{thm}

\begin{proof}
The Chern numbers of the minimal desingularization $Y$ of $X$ are: $c_{1}^{2}= d(d-4)^{2}$, 
$c_{2}= d(d^2-4d+6)$.
The Chern numbers of the orbifold $\mathcal{X}$ are $c_{1}^{2}(\mathcal{X})= c_1^2$ and 
$$c_{2}(\mathcal{X})= c_2-\left(d-\frac{1}{d}\right) \left(\sum_{i<j} d_id_j\right).$$
Thus 
\[
\begin{array}{cc}
s_{2}(Y)+s_{2}(\mathcal{X})&=4d(5-2d)+\left(d-\frac{1}{d}\right) \left(\sum_{i<j} d_id_j\right)\\
&>  4d(5-2d)+\frac{k(k-1)}{2} \left(d-\frac{1}{d}\right)c^2 .
\end{array}
\] \end{proof}

As a corollary, we obtain many examples of surfaces in $\bP^3$
with big cotangent bundle.

\begin{example}
For every $d\geq 15$, the degree $d$ covering of $d$ lines in $\bP^2$ has big cotangent sheaf.
\end{example}

\subsubsection{Criteria when the $D_j$ are lines and for $n$ dividing $d$}
 Let us consider the case when all the curves $D_{i}$ are lines and the degree $n$ of the cover divides $d$.
Thus $d=k$ and the number $\ell$ of $A_{n-1}$ singularities is
$d(d-1)/2.$ Since $s_{2}(Y)+s_{2}(\mathcal{X})=2s_{2}(X)+\ell(n-\frac{1}{n})$,
we get\[
s_{2}(Y)+s_{2}(\mathcal{X})=2n[6-(n-1)(3v+v^{2})]+\frac{1}{2}(nv-1)(n^{2}-1)v,\] 
where $d=nv$. For a cover of degree $n=2$ and $n=3$, we get respectively $s_{2}(X)+s_{2}(\mathcal{X})=24-\frac{27}{2}-v^{2}$
and $s_{2}(X)+s_{2}(\mathcal{X})=36-40v$ ; this is always negative
(for $v\geq2$) and we cannot apply Theorem \ref{main}. For the remaining
cases, a simple computation gives:
\begin{thm}
For a cyclic cover of degree $n\geq4$ branched over the union of $d=vn>4$
lines in general position, we have $s_{2}(Y)+s_{2}(\mathcal{X})>0$
except for the following finite number of cases for the couples $(v,n)$:

$$\begin{tabular}{l|ccccccc}
  \hline
  v &  $1$ & $2$ & $3$ & $4\leq v\leq6$ & $7\leq v\leq12$  \\
  \hline
  $n$ & $\leq14$ & $\leq8$ & $\leq6$ & $4,5$ & $4$  \\
  \hline
\end{tabular}$$

\end{thm}

\subsection{Remarks when $s_2(Y)>0$}

We close this section by remarking that Theorem \ref{main} has also an application to surfaces with $s_2(Y)>0$.
Let us consider a surface $Y$ with ample cotangent bundle. For $m>>0$, we have $h^i(Y,S^m \Omega_Y)=0,\,i=1,2$, and thus
 $$h^0(Y,S^m \Omega_Y)=\frac{s_2(Y)}{6}\, m^3 +O(m^2).$$
 Suppose that $Y$ has a deformation  $Y_0$ which is a surface containing one $(-2)$-curve (examples of such surfaces can be obtained e.g. as complete intersection of ample divisors in an Abelian variety). Then the space of symmetric differentials jumps :
 $$h^0(Y_0,S^m \Omega_{Y_0}) \geq  (\frac{s_2(Y)}{6}+\frac{1}{8})\, m^3+O(m^2) >h^0(Y,S^m \Omega_Y).$$
This is another illustration of the importance of the presence or absence of $(-2)$-curves for the geometry of a surface.


\section{On the geography of the surfaces with big cotangent bundle}
\subsection{A Chern classes inequality}

As mentioned in the introduction, surfaces of general type with ample cotangent bundle are known to satisfy the Chern classes inequality $c_1^2 > c_2$ \cite{FL}. A natural question is to ask if surfaces of general type with big cotangent bundle satisfy a Chern classes inequality. We investigate here the case of a surface that satisfies the hypothesis of Theorem \ref{main}.

\begin{prop}\label{ineq}
The Chern numbers of a surface $Y$ that satisfies the hypothesis
of Theorem \ref{main} (i.e. $s_{2}(Y)+s_{2}(\mathcal{X})>0$) are such that:
\begin{equation}\label{cci}
c_{1}^{2}(Y) > \frac{3}{5}c_{2}(Y).
\end{equation}

\end{prop}

\begin{proof}
Since $s_{2}(Y)+s_{2}(\mathcal{X})>0$, dividing this inequality by  $c_{1}^{2}(\mathcal{X})=c_{1}^{2}(Y)$, we obtain
$$2-\frac{c_2(Y)}{c_1 ^2(Y)}-\frac{c_2(\mathcal X)}{c_1 ^2(Y)}>0.$$
One of the main result of \cite{Miy} translated into the language of orbifolds is the orbifold Bogomolov-Miyaoka-Yau
inequality for surfaces with quotient singularities:
$$c_{1}^{2}(\mathcal{X})\leq3c_{2}(\mathcal{X}).$$
Applying this inequality, we obtain $$\frac{c_2(\mathcal X)}{c_1 ^2(Y)}>\frac{1}{3}$$ and  $$\frac{c_2(X)}{c_1 ^2(Y)}<\frac{5}{3}.$$

\end{proof}

\begin{remark}
As mentioned above, the existence of a surface of degree $10$ in $\bP^3$ with big cotangent sheaf is obtained in \cite{Thomas}. This shows that the inequality \ref{cci} is not satisfied by all such surfaces.
\end{remark}

\begin{remark}
Recall that for minimal surfaces of general type, we have the Noether inequality : $$c_{1}^{2}(Y)\geq \frac{1}{5}(c_{2}-36).$$ The surfaces that are on the Noether line $c_{1}^{2}=\frac{1}{5}(c_{2}-36)$ are called Horikawa surfaces. As a consequence, it is hopeless to get the proof in that way of the existence of (higher) symmetric forms for Horikawa surfaces with canonical singularities.
\end{remark}

\begin{remark}\label{sumchernratio}
Since we suppose $s_{2}(Y)+s_{2}(\mathcal{X})>0$, an immediate computation gives
$$ \frac{c_2 (Y)+c_2(\mathcal X)}{c_1 ^2(Y)} =  \frac{c_2 (Y)}{c_1 ^2(Y)}+\frac{c_2 (\mathcal{X})}{c_1 ^2 (\mathcal{X})}<2.$$ 
Therefore a ratio $\frac{c_1 ^2 (Y)}{c_2(Y)}$ close to $3/5$ forces the ratio $\frac{c_1 ^2 (\mathcal{X})}{c_2 (\mathcal{X})} $ to be close to the Miyaoka bound $3$.  Examples of orbifolds with $s_2(Y)<0$ and $\frac{c_1 ^2 (\mathcal{X})}{c_2 (\mathcal{X})} $ close to $3$ are rather rare, see \cite{RR} for some references.

\end{remark}

\bigskip

\noindent 
Xavier Roulleau,\\  
{\tt roulleau@math.univ-poitiers.fr}\\
Universit\'e de Poitiers,\\
Laboratoire de Math\'ematiques et Applications\\
Avenue du T\'el\'eport 1,\\
86000 Poitiers\\
France

\bigskip

\noindent 
Erwan Rousseau,\\
{\tt erwan.rousseau@cmi.univ-mrs.fr}\\ 
Laboratoire d'Analyse, Topologie, Probabilit\'es\\ 
Universit\'e d'Aix-Marseille et CNRS\\ 
39, rue Fr\'ed\'eric Joliot-Curie\\ 
13453 Marseille Cedex 13\\ 
France 
\end{document}